\documentclass[a4paper,11pt]{amsart}
\usepackage{amsmath,amssymb,amsfonts,amsthm}
\usepackage{latexsym,mathrsfs}
\usepackage{graphicx}
\usepackage[all]{xy}
\input xypic
\usepackage{color}
\usepackage[backref=false]{hyperref}
\hypersetup{colorlinks=true,linkcolor=red,citecolor=blue}
\usepackage[OT2,T1]{fontenc}
\usepackage[utf8]{inputenc}
\usepackage{lmodern}
\usepackage{microtype}
\usepackage{enumerate}
\usepackage{tikz-cd}
\usepackage{bm}
\theoremstyle{plain}
\newtheorem{theorem}[subsection]{{\bf Theorem}}
\newtheorem*{theorem*}{{\bf Theorem}}
\newtheorem{corollary}[subsection]{{\bf Corollary}}
\newtheorem*{corollary*}{{\bf Corollary}}
\newtheorem{proposition}[subsection]{{\bf Proposition}}
\newtheorem{lemma}[subsection]{{\bf Lemma}}

\theoremstyle{definition}

\theoremstyle{remark}

\numberwithin{equation}{subsection}

\DeclareBoldMathCommand{\bbot}{\bot}

\DeclareSymbolFont{cyrletters}{OT2}{wncyr}{m}{n}
\DeclareMathSymbol{\Sha}{\mathalpha}{cyrletters}{"58}
\DeclareMathOperator{\pwc}{pwc}
\DeclareMathOperator{\pwh}{pwh}

\begin{document}
\title[Powerful class]{The powerful class of Sylow subgroups of finite groups}
\author{Primo\v z Moravec}
\address{{
Faculty of  Mathematics and Physics, University of Ljubljana,
and Institute of Mathematics, Physics and Mechanics,
Slovenia}}
\email{primoz.moravec@fmf.uni-lj.si}
\thanks{ORCID: \url{https://orcid.org/0000-0001-8594-0699}. The author acknowledges the financial support from the Slovenian Research Agency, research core funding No. P1-0222, and projects No. J1-3004, N1-0217, J1-2453, J1-1691.}
\date{\today}
\begin{abstract}
\noindent
The paper explores the effect of powerful class of Sylow $p$-subgroups of a given finite group on control of transfer or fusion. We also find an explicit bound for the $p$-length of a $p$-solvable group in terms of the poweful class of a Sylow $p$-subgroup. 
\end{abstract}
\maketitle
\section{Introduction}
\label{s:intro}

\noindent
The class of powerful $p$-groups was introduced by Lubotzky and Mann in their landmark paper \cite{LM87}. Powerful $p$-groups have found numerous applications in the quest of understanding the structure of finite $p$-groups. Mann \cite{Man95} also proved that if a finite group $G$ has a powerful Sylow $p$-subgroup $P$, then $N_G(P)$ controls transfer in $G$.

Related to being powerful is the notion of powerfully embedded subgroup of a finite $p$-group. Every finite $p$-group $P$ has the largest powerfully embedded subgroup $\eta(P)$ by \cite[Theorem 1.1]{LM87}. It is not difficult to see that $\eta$ is a positive characteristic $p$-functor. 

Our first main result is a variant of the second Gr\" un theorem for $\eta$:

\begin{theorem}
    \label{thm:2ndgrun}
    Let $G$ be a finite group and $P$ a Sylow $p$-subgroup of $G$. If $\eta(P)$ is weakly closed in $P$ with respect to $G$, then $P\cap G'=P\cap N_G(\eta (P))'$.
\end{theorem}

A direct consequence of Theorem \ref{thm:2ndgrun} is that, under the assumptions of the theorem, the functor $\eta$ controls transfer in $G$ (Corollary \ref{cor:etatransfer}). In Section \ref{s:local} we also indicate when $\eta$ strongly controls fusion in $G$.

The second part of the paper deals with the influence of the powerful class of a Sylow $p$-subgroup on the structure of a given finite group. The notion of powerful class $\pwc P$ of a finite $p$-group $P$ was first defined by Mann \cite{Man11}, and further explored in \cite{Mor24}. The group $P$ is powerful if and only if $\pwc P\le 1$. Powerful class is a refinement of the nilpotency class, as the nilpotency class of $P$ provides an upper bound for $\pwc P$. Finite $p$-groups $P$ with $\pwc P<p$ are called the {\it $p$-groups of small powerful class}. These groups have well-behaved power structure \cite{Man11}. We show that finite $p$-groups of small powerful class also behave well as Sylow $p$-subgroups of a finite group. More precisely, if a finite group $G$ has a Sylow $p$-subgroup $P$ of small powerful class, then $N_G(P)$ controls transfer in $G$ (Proposition \ref{prop:Ngtransfer}). This generalizes the main result of \cite{Man95} mentioned above. 

Focusing on $p$-solvable groups $G$, we consider how the powerful class of a Sylow $p$-subgroup $P$ affects the $p$-length $\ell_p(G)$ of $G$. Our starting point is a result of Khukhro \cite[Theorem 4.1]{Khu12} showing that if $P$ is powerful, then $\ell_p(G)\le 1$. We generalize this as follows.

\begin{theorem}
    \label{thm:plength}
    Let $G$ be a finite $p$-solvable group and $P$ a Sylow $p$-subgroup of $G$. 
    \begin{enumerate}
        \item If $p>2$, then $\ell_p(G)\le \left\lceil \pwc P / (p-2)\right\rceil$.
        \item If $p=2$, then $\ell_2(G)\le \pwc P$.
    \end{enumerate}
\end{theorem}

Note that Theorem \ref{thm:plength} also refines and improves the classical result that $\ell_p(G)$ is bounded from above by the nilpotency class of $P$.

\section{Preliminaries}
\label{s:prelim}

\noindent
A normal subgroup $N$ of a finite $p$-group $P$ is {\it powerfully embedded} in $P$ if $[N,P]\le N^{2p}$. If $P$ is powerfully embedded in itself, we say that $P$ is a {\it powerful $p$-group}. 

Denote by $\eta(P)$ the largest powefully embedded subgroup of $P$, that is, the product of all powerfully embedded subgroups of $P$. Clearly, $\eta (P)$ is a characteristic subgroup of $P$ and contains the center $Z(P)$ of $P$.

According to Mann \cite{Man11}, an ascending series
$1=N_0\le N_1\le N_2\le \cdots$
of normal subgroups of $P$ is said to be an {\it $\eta$-series} if $N_{i+1}/N_i$ is powerfully embedded in $P/N_i$ for all $i$. The shortest length $k$ of an $\eta$-series with $N_k=N$ is called the {\it powerful height} of $N$. It is denoted by $\pwh_P(N)$. 

The number $\pwc(P)=\pwh_P(P)$ is called the {\it powerful class} of $P$. It is easy to see that if $K\triangleleft P$, then $\pwc(P/K)\le \pwc P$.

The {\it upper $\eta$-series of $P$} is defined by $\eta_0(P)=1$ and $$\eta_{i+1}(P)/\eta_i(P)=\eta(P/\eta_i(P))$$
for $i\ge 0$. 
 It is straightforward to see \cite[Proposition 2.1]{Mor24} that the upper $\eta$-series of $P$ is the fastest growing $\eta$-series in $P$.

A group $P$ is said to have {\it small powerful class} if $\pwc(P)<p$. Similarly, a normal subgroup $N$ of $P$ has {\it small powerful height} if $\pwh_P(N)<p$.

\begin{lemma}
    \label{lem:pwh}
    Let $p>2$ and let $M$ and $N$ be normal subgroups of a finite $p$-group $P$.
    \begin{enumerate}
        \item $\pwh_P[N,P]\le \pwh_PN$,
        \item $\pwh_P N^p\le \pwh_PN$,
        \item $\pwh_P MN\le\max\{\pwh_PM,\pwh_PN\}$.
    \end{enumerate}
\end{lemma}

\begin{proof}
    Items (1) and (2) are proved in \cite[Lemma 2.5]{Man11}. To prove (3), take the $\eta$-series $1=N_0\le N_1\le\cdots \le N_k=N$ and $1=M_0\le M_1\le\cdots \le M_\ell=M$ of $N$ and $M$ in $P$, respectively. Without loss of generality we may assume that $k\ge \ell$, and we put $M_i=M$ for $i\ge \ell$. Then 
    $[M_{i+1}N_{i+1},P]=[M_{i+1},P][N_{i+1},P]\le M_{i+1}^pM_iN_{i+1}^pN_i\le (M_{i+1}N_{i+1})^p(M_iN_i)$, therefore $1=M_0N_0\le M_1N_1\le\cdots\le M_kN_k=MN$ is an $\eta$-series of $MN$ in $P$.
\end{proof}

\begin{lemma}
    \label{lem:pwhNp}
    Let $p$ be odd, and let $N$ be a normal subgroup of a finite $p$-group $P$. If $\pwh_PN\le j$, then $[N,{}_jP]\le N^p$.
\end{lemma}

\begin{proof}
    This is true for $j=1$, as $\pwh_PN=1$ implies that $N$ is powerfully embedded in $P$. Assume the claim holds for some $j\ge 1$. Let $\pwh_PN=j+1$, and let
    $1=N_0\le N_1\le \cdots \le N_{j+1}=N$ be an $\eta$-series in $P$. Then $[N,P]\le N^pN_j$. As $\pwh_PN_j\le j$, we have that $[N_j,{}_jP]\le N_j^p$. Thus
    $[N,{}_{j+1}P]=[N^p,{}_jP][N_j,{}_jP]\le N^pN_j^p=N^p$.
\end{proof}

Other notions are standard \cite{HupI,HBIII}.
\section{Local control}
\label{s:local}

\noindent
In this section we prove the results related to the local control. At first we note the following:

\begin{lemma}
    \label{lem:cpwrcp}
    Let $p$ be a prime and $P=C_p\wr C_p$. Then $\eta_i(P)=Z_i(P)$ for all $i\ge 0$.
\end{lemma}

\begin{proof}
    If $p=2$, this is straightforward to verifiy. Assume $p>2$. Note that $P$ is a $p$-group of maximal class. The result now follows immediately from \cite[Corollary 5.8]{Mor24}.
\end{proof}

This provides our first results on control of transfer:

\begin{proposition}
    \label{prop:Ngtransfer}
    Let $G$ be a finite group and $P$ a Sylow $p$-subgroup of $G$. If $P$ has small powerful class, then $N_G(P)$ controls transfer in $G$.
\end{proposition}

\begin{proof}
    By Lemma \ref{lem:cpwrcp}, the powerful class of $C_p\wr C_p$ is equal to $p$. It follows that $C_p\wr C_p$ cannot be a quotient of $P$. The result now follows from Yoshida's transfer theorem \cite[10.1 Theorem]{Isa11}.
\end{proof}

A consequence is the following:

\begin{proposition}
    \label{prop:Ngpnilpotent}
    Let $G$ be a finite group and $P$ a Sylow $p$-subgroup of $G$. If $P$ has small powerful class, and if $N_G(P)$ is $p$-nilpotent, then $G$ is $p$-nilpotent.
\end{proposition}

\begin{proof}
    Denote $N=N_G(P)$. As $N$ is $p$-nilpotent, we have that $O^p(N)\cap P=1$. By Proposition \ref{prop:Ngtransfer}, we also have that $O^p(N)=N\cap O^p(G)$. This gives $O^p(G)\cap P=1$, hence the result.
\end{proof}

The second Gr\" un theorem \cite[IV, 3.7]{HupI} states that if a finite group $G$ is $p$-normal and if $P$ is a Sylow $p$-subgroup of $G$, then $P\cap G'=P\cap N_G(Z(P))'$. Theorem \ref{thm:2ndgrun} states that an analogous result holds when $Z(P)$ is replaced with $\eta(P)$. The proof goes as follows:

\begin{proof}[Proof of Theorem \ref{thm:2ndgrun}]
    Denote $W=\eta(P)$. By Gr\" un's first theorem \cite[IV, 3.4 Satz]{HupI}, we have that
    $P\cap G'=\langle P\cap N_G(P)',P\cap (P')^g\mid g\in G\rangle$. Thus it suffices to show that $P\cap (P')^g\subseteq P\cap N_G(W)'$ for every $g\in G$.
    
    Denote $T=P\cap (P')^g$. Let $i$ be the smallest non-negative integer with the property that $W^{p^i}\subseteq N_G(T)$. Suppose that $i>0$. Then there exists $x\in W^{p^{i-1}}$ and there exists $t\in T$ such that $t^x\notin T$. Then $[t,x]\in [P,W^{p^{i-1}}]\le W^{p^i}\le T$, therefore $t^x=t[t,x]\in T$, a contradiction. This shows that $W\subseteq N_G(T)$. We also have $W^g\subseteq N_G(T)$ by a similar argument. Thus there exist a Sylow $p$-subgroup $P_1$ of $N_G(T)$  and $n\in N_G(T)$ such that $W\subseteq P_1$ and $W^g\subseteq P_1^n$. Let $P_2$ be a Sylow $p$-subgroup of $G$ that contains $P_1$. Then $W\subseteq P_2$ and $W^{gn^{-1}}\subseteq P_2$. As $W$ is weakly closed in $P_2$ with respect to $G$, we have that
    $W=W^{gn^{-1}}$. This shows that $gn^{-1}\in N_G(W)$. We have $T=T^{n^{-1}}=P^{n^{-1}}\cap (P')^{gn^{-1}}$. As $P\subseteq N_G(W)$, we conclude that 
    $(P')^{gn^{-1}}\subseteq (N_G(W)')^{gn^{-1}}=N_G(W)'$, hence 
    $T=P\cap (P')^g=P^{n^{-1}}\cap (P')^{gn^{-1}}\subseteq P\cap N_G(W)'$. This concludes the proof.
\end{proof}

\begin{corollary}
    \label{cor:etatransfer}
    Let $G$ be a finite group and $P$ a Sylow $p$-subgroup of $G$. If $\eta(P)$ is weakly closed in $P$ with respect to $G$, then the functor $\eta$ controls transfer in $G$.
\end{corollary}

\begin{proof}
    This follows directly from Theorem \ref{thm:2ndgrun} and \cite[p. 52]{HBIII}.
\end{proof}

Corollary \ref{cor:etatransfer} is related to a result of Gilloti and Serena \cite{GS84} stating that if $\Phi(P)$ is strongly closed in $P$ with respect to $G$, then $\Phi$ controls transfer in $G$. In the case when $\eta (P)$ is strongly closed in $P$ with respect to $G$, we can say more:

\begin{proposition}
    \label{prop:fusion}
    Let $G$ be a finite group and $P$ a Sylow $p$-subgroup of $G$. If $\eta(P)$ is strongly closed in $P$ with respect to $G$, and all $\eta(P)^{p^i}$ are weakly closed in $P$ with respect to $G$, then $\eta$ strongly controls fusion in $G$.
\end{proposition}

\begin{proof}
    As $\eta(P)\ge \eta(P)^p\ge\eta(P)^{p^2}\ge \cdots$ forms a central series of $\eta(P)$, the claim follows from the main result od Gilloti and Serena in \cite{GS85}.
\end{proof}

We note here that if, for example, $G$ is a $p$-nilpotent group with a Sylow $p$-subgroup $P$, then every normal subgroup of $P$ is strongly closed in $P$ with respect to $G$.
\section{Powerful class and $p$-length}
\label{s:plength}

\noindent
In this section we prove Theorem \ref{thm:plength}. 

We recall a notion introduced by Gonz\' alez-S\' anchez \cite{Gon07}.
Let $P$ be a finite $p$-group. A sequence $N=N_1\ge N_2\ge\cdots\ge N_k=1$ of normal subgroups of $P$ is a {\it potent filtration of type $t$} if $[N_i,P]\le N_{i+1}$ and $[N_i,{}_tP]\le N_{i+1}^p$ for all $i=1,\ldots ,k-1$. In this case we say that $N$ is {\it PF-embedded of type $t$ in $P$}. We have the following results proved by Gonz\' alez-S\' anchez and Spagnuolo \cite{GSS15}:

\begin{lemma}[\cite{GSS15}]
    \label{lem:GSS}
    Let $G$ be a $p$-solvable group, $P$ its Sylow $p$-subgroup and $N$ a normal subgroup of $P$.
    \begin{enumerate}
        \item If $p>2$ and $N$ is PF-embedded of type $p-2$ in $P$, then $N\subseteq O_{p'p}(G)$.
        \item If $p=2$ and $N$ is PF-embedded of type $1$ in $P$, then $N\subseteq O_{2'2}(G)$.
    \end{enumerate}
\end{lemma}

At first we mention a variation of Khukhro's result \cite[Theorem 4.1]{Khu12}:

\begin{proposition}
    \label{prop:peO}
    Let $G$ be a finite $p$-solvable group and $P$ a Sylow $p$-subgroup of $G$. Let $N\triangleleft P$ be powerfully embedded in $P$. Then $N\subseteq O_{p'p}(G)$.
\end{proposition}

\begin{proof}
    Consider first the case $p=2$. Here we note that the series $N\ge N^2\ge N^4\ge N^8\ge \cdots$ is a potent filtration of type $1$ starting with $N$. This shows that $N$ is PF-embedded of type $1$ in $P$. The result now follows from item (2) of Lemma \ref{lem:GSS}.

    We are left with the case when $p$ is odd.
    The proof is essentially an adaptation of the proof of \cite[Theorem 4.1]{Khu12}, we give it for reader's convenience. Let $G$ be a smallest possible counterexample, with $N$ powerfully embedded in $P$ and $N\not\subseteq O_{p'p}(G)$. By minimality, we have that $O_{p'}(G)=1$, and we may assume that $V=O_p(G)$ is an elementary abelian $p$-group with $N\not\subseteq V$. By \cite[Theorem 6.3.2]{Gor80} we have that $C_G(V)=1$, and thus $V$ is an $\mathbb{F}_p[G/V]$-module.

    If $Q$ is Hall $p'$-subgroup of $O_{pp'}(G)$, then we may conclude as in \cite{Khu12} that $Q$ acts faithfully on $V$, and we have that $O_{pp'}(G)=VQ$ and $G=VN_G(Q)$. Furthermore, if $S$ is a Sylow $p$-subgroup of $N_G(Q)$, then $P=V\rtimes S$.
    
    Take $x\in N$ of maximal possible order $p^n$. If $n=1$, then $[N,P]\le N^p=1$, therefore $N\le Z(P)\le V$. So we have that $n\ge 2$. Let $y=x^{p^{n-2}}$. As $[N^{p^{n-1}},P]\le N^{p^n}=1$, it follows that $y^p\in Z(P)\le V$. Note also that $y$ does not belong to $V$, as its order is $p^2$. On the other hand, $[N^{p^{n-2}},P,P]\le N^{p^n}=1$ implies $y\in Z_2(P)$. Now write $y=vs$, where $v\in V$ and $s\in S\setminus\{ 1\}$. As $s^p\equiv y^p\mod [V,N]$, we conclude that $s^p\in V\cap S=1$. This implies that the linear transformation of $V$ induced by conjugation by $s$ has at least one Jordan block of size $p\times p$. Thus there exists $w\in V$ with $[w,_{p-1}s]\neq 1$. This gives $[w,{}_{p-1}y]=1$, which contradicts the fact that $y\in Z_2(P)$.
\end{proof}

Our next result provides a basis for induction used later on in the proof of Theorem \ref{thm:plength} for the case $p>3$.

\begin{proposition}
    \label{prop:PFtypep-2}
    Let $G$ be a finite $p$-solvable group and $P$ a Sylow $p$-subgroup of $G$, where $p>3$. Let $N\triangleleft P$ satisfy $\pwh_PN<p-1$. Then $N\subseteq O_{p'p}(G)$.
\end{proposition}

\begin{proof}
    By Lemma \ref{lem:GSS} (1), it suffices to show that $N$ is PF-embedded of type $p-2$ in $P$. Let 
    $$1=N_0\le N_1\le \cdots \le N_k=N$$
    be an $\eta$-series of $N$ in $P$ with $k<p-1$. Put $N_j=1$ for $j\le 0$, and $N_j=N$ for $j\ge k$.

    At first we observe that Lemma \ref{lem:pwhNp} gives that if $K$ is any normal subgroup of $P$ with $\pwh_PK\le p-2$, then $[K,{}_{p-2}P]\le K^p$ for all $j$. This, in particular, implies that $[K^p,P]=[K,P]^p$ and $(K^{p^i})^{p^j}=K^{p^{i+j}}$by \cite[Proposition 3.2, Theorem 3.4]{FGJ08}. This, together with Lemma \ref{lem:pwh}, will be used throughout the proof without further mention.
    
    Define 
    \begin{align*}
        M_1 &= N,\\
        M_{i+1} &= M_i^pN_{p-i-2},
    \end{align*}
    where $i>0$. We claim that $(M_i)_i$ is a potent filtration of type $p-2$ in $P$.

    At first note that $[M_1,P]=[N_1,P]\le N_1^pN_2\le N_1^pN_{p-3}=M_2$, since $p\ge 5$. Suppose $[M_i,P]\le M_{i+1}$ for some $i\ge 1$. Then
    $[M_{i+1},P]=[M_i^pN_{p-i-2},P]=[M_i,P]^p[N_{p-i-2},P]\le M_{i+1}^pN_{p-i-2}^pN_{p-i-3}=M_{i+1}^pN_{p-i-3}=M_{i+2}$.

    It remains to show that $[M_i, {}_{p-2}P]\le M_{i+1}^p$ for all $i\ge 0$. At first consider $[M_1,{}_{p-2}P]=[N,{}_{p-2}P]$. By the previous paragraph, $[M_1,{}_{p-2}P]\le M_{p-1}=M_{p-2}^pN_{p-(p-2)-2}=M_{p-2}^p\le M_2^p$, as required. Assume the claim holds for some $i\ge 1$. We have $[M_{i+1},{}_{p-2}P]=[M_i^p,{}_{p-2}P][N_{p-i-2},{}_{p-2}P]=[M_i,{}_{p-2}P]^p[N_{p-i-2},{}_{p-2}P]\le M_{i+1}^{p^2}[N_{p-i-2},{}_{p-2}P]$. We need to show that the latter group is contained in $M_{i+2}^p$.  We may assume $M_{i+2}^p=1$, which yields $M_{i+1}^{p^2}=N_{p-i-3}^p=1$. Thus the problem reduces to showing that, under these assumptions, $[N_{p-i-2},{}_{p-2}P]=1$. To see this, first observe that $M_{i+1}^{p^2}=1$ implies $N_{p-i-2}^{p^2}=1$.
    This implies $[N_{p-i-2},P,P]\le [N_{p-i-2}^pN_{p-i-3},P]= [N_{p-i-2},P]^p[N_{p-i-3},P]\le N_{p-i-2}^{p^2}N_{p-i-3}^pN_{p-i-4}=N_{p-i-4}$. Inductive argument implies that
    $[N_{p-i-2},{}_\ell P]\le N_{p-i-\ell -2}$ for all $\ell\ge 2$. As $p-2\ge 2$, we conclude $[N_{p-i-2},{}_{p-2}P]\le N_{p-i-(p-2)-2}=1$. This concludes the proof.
\end{proof}

\begin{corollary}
    \label{cor:etap-2}
    Let $G$ be a finite $p$-solvable group and $P$ a Sylow $p$-subgroup of $G$. If $p>2$, then $\eta_{p-2}(P)\subseteq O_{p'p}(G)$.
\end{corollary}

\begin{proof}
    If $p=3$, this follows from Proposition \ref{prop:peO}, as $\eta_1(P)=\eta(P)$ is powerfully embedded in $P$. For $p>3$ note that $\pwh_P \eta_{p-2}(P)\le p-2$ by \cite[Lemma 2.3]{Mor24}, hence the result follows from Proposition \ref{prop:PFtypep-2}.
\end{proof}

\begin{proof}[Proof of Theorem \ref{thm:plength}]
    Suppose first $p=2$. We prove the claim by induction on $k=\pwc P$. If $k=1$, the result follows from \cite[Theorem 4.1]{Khu12}. Suppose the claim holds for some $k\ge 1$, and let $\pwc P=k+1$, Then Lemma \ref{lem:GSS} (2) implies $\pwc(PO_{p'p}(G)/O_{p'p}(G))\le \pwc (P/\eta (P))=\pwc P-1=k$, hence $\ell_p(G)=\ell_p(G/O_{p'p}((G))+1\le k+1$.

    In the case when $p$ is odd, we proceed by induction on the smallest integer $k$ with the property that $k(p-2)\ge \pwc P$. If $k=1$, then the claim follows from Corollary \ref{cor:etap-2}. For the inductive step note that $\pwc(PO_{p'p}(G)/O_{p'p}(G))\le \pwc (P/\eta_{p-2} (P))$, and repeat the argument in the previous paragraph.
\end{proof}

\end{document}